\documentclass[11pt,reqno]{amsart}
\usepackage{float}
\usepackage{graphicx}
\usepackage{amsfonts,color,amsmath,amssymb,amsthm,fancyhdr}
\usepackage{makecell}
\usepackage{multirow}
\usepackage{footnote}
\usepackage{appendix}
\usepackage{hyperref}

\newcommand{\ra}{\rangle}
\newcommand{\la}{\langle}

\newcommand{\F}{{\mathbb F}}

\newcommand{\bQ}{{\textbf Q}}
\newcommand{\cP}{{\mathcal P}}

\newcommand{\cS}{{\mathcal S}}
\newcommand{\cD}{{\mathcal D}}
\newcommand{\cL}{{\mathcal L}}

\newcommand{\bR}{{\mathbb R}}

\newcommand{\GL}{\textup{GL}}

\newcommand{\GamL}{\Gamma\textup{L}}

\newcommand{\Sp}{\textup{Sp}}

\newcommand{\lcm}{\textup{lcm}}

\newtheorem{thm}{Theorem}

\newtheorem{lemma}[thm]{Lemma}
\newtheorem{corollary}[thm]{Corollary}

\newtheorem{proposition}[thm]{Proposition}

\numberwithin{equation}{section}
\numberwithin{thm}{section}

\newtheorem{remark}[thm]{Remark}

\begin{document}
\title{On finite generalized quadrangles of even order}

\author[Tao Feng]{Tao Feng}
\thanks{Tao Feng is with School of Mathematical Sciences, Zhejiang University, Hangzhou 310027, Zhejiang, P.R. China, (e-mail: tfeng@zju.edu.cn).}
\maketitle

\begin{abstract}
In this paper, we establish the following two results: (1) a skew translation generalized quadrangle of even order is a translation generalized quadrangle, (2) a generalized quadrangle of even order does not admit a point regular group of automorphisms. The first result confirms a conjecture of Payne (1975) based on earlier work of Ott (2021), and the second result confirms a conjecture of Ghinelli (1992).
\end{abstract}

\section{Introduction}

A generalized quadrangle $\cS$ of order $(s,t)$ is a finite point-line incidence structure such that each line has $s+1$ points, each point is on $t+1$ lines, and for each point-line pair $(P,\ell)$ that is not incident there is a unique point collinear with $P$ on $\ell$.   If $s=t$, we say that it has order $s$. The quadrangle $\cS$ is said to be thick if $\min\{s,t\}\ge 2$. An automorphism of $\cS$ consists of a permutation of the points and a permutation of lines that preserve the incidence. For a given point $P$, an elation about $P$ is an automorphism $\theta$ that is either an identity, or fixes each line through $P$ and no point not collinear with it. If further $\theta$ fixes each point collinear with $P$, then it is a symmetry about $P$. We refer the reader to \cite{FGQ} for basics on generalized quadrangles and \cite{sym} for a comprehensive study from the viewpoint of symmetries.

Suppose that $G$ is a group of automorphisms of a generalized quadrangle $\cS$. If $G$ consists of elations about a point $P$ and acts regularly on the points not collinear with $P$, then $\cS$ is an elation generalized quadrangle with base point  $P$ and elation group $G$. If the elation group $G$ contains a subgroup consisting of $t$ symmetries about $P$, then $\cS$ is a skew translation generalized quadrangle. If an elation quadrangle $\cS$ has an abelian elation group, then it is a translation generalized quadrangle. We refer to \cite{kantor} for a description  of elation generalized quadrangles using coset geometries, and refer to \cite{central,tgq} for more on skew-translation and translation quadrangles.

The classification of finite generalized quadrangles that satisfy certain transitivity assumptions has a long history. In recent years, there has been significant progress towards Kantor's 1991 conjecture \cite{kt91} that a finite flag-transitive generalized quadrangle is either classical or has order $(3,5)$ or $(15,17)$ up to duality, cf. \cite{antiflag,anti2}. The generalized quadrangles with a group of automorphisms acting primitively on points have also attracted much attention, cf. \cite{bppprim} and the references therein. In the O'Nan-Scott Theorem as stated in \cite[Section~6]{perm}, there are several types of  primitive permutation group actions that have a regular subgroup.

In this paper, we are concerned with skew-translation generalized quadrangles and quadrangles with a point regular group of automorphisms. The symplectic quadrangle $W(q)$ whose points and lines are the totally singular points and lines respectively of the rank $2$ symplectic polar space $W(3,q)$ is a skew translation generalized quadrangle, and it is a translation generalized quadrangle when $q$ is even. Conversely, a skew-translation generalized quadrangle of odd order $q$ must be the symplectic quadrangle $W(q)$. This follows by combining the results in Yoshiara's 2007 work \cite{Yreg} and Ghinelli's 2012 work \cite{DG}, and is first observed by Bamberg, Glasby and Swartz \cite{bgwq8} who also corrected an error in \cite{DG}. K. Thas \cite{central} also has an independent proof of this fundamental characterization. In the even order case, Payne \cite{Payne75} made in 1975 the conjecture that a skew translation generalized quadrangle of even order $s$ must be a translation generalized quadrangle. The order $s$ of a skew translation generalized quadrangle  is a power of $2$ and the elation group $G$ has order  $s^3$, cf. \cite{chen,Hachen}. In 2021, Ott \cite{ott} made major progress and confirmed the conjecture in the case where $s$ is an odd power of $2$. In  the preprint \cite{KtPre}, K. Thas gave a short geometrical proof of Ott's result.

The study of finite generalized quadrangles with a regular group of automorphisms was initiated by Ghinelli \cite{Greg} in 1992, and we refer to \cite{BGreg,SKreg,feng,Swreg,Yreg} for research on this topic.  By determining all the point regular groups of automorphisms of the Payne derived quadrangle of the symplectic quadrangle $W(q)$, $q$ odd, we \cite{feng} showed that the finite groups that act regularly on the points of a finite generalized quadrangle can have unbounded nilpotency classes. This is an indication of difficulty in the study of such quadrangles. In Ghinelli's original work \cite{Greg}, she made the conjecture that a generalized quadrangle of even order does not admit a regular group of automorphisms.  In \cite{Swreg}, Swartz made the milder conjecture that a thick generalized quadrangle of order $s$ can not have a group of automorphisms acting regularly on both the point set and the line set. By Theorem 1.3 of \cite{Swreg}, a counterexample would have even order. There has been little progress on both conjectures.

In this paper, we confirm both Payne's 1975 conjecture and Ghinelli's 1992 conjecture by introducing new techniques from character theory and extremal graph theory. Swartz's conjecture in \cite{Swreg} then follows as a corollary. To be specific, we prove the following results.
\begin{thm}\label{thm_STGQ}
A skew translation generalized quadrangles of even order is a translation generalized quadrangle.
\end{thm}

\begin{thm}\label{thm_regGQ}
A generalized quadrangle of even order does not admit a point regular group of automorphisms.
\end{thm}

Here is an outline of the paper. For an elation generalized quadrangle $\cS$ of order $(s,t)$ with elation group $G$, we define a $4$-dimensional algebra from the associated $4$-gonal family. We use it to define two class functions $\chi_S,\chi_T$ on $G$ that bear geometric information about $\cS$, and show that they are characters of $G$ in Section 2.1. In the case $\cS$ is a skew translation quadrangle of even order $s$, we apply the characters to provide an alternative proof for the part of Ott's proof in \cite{ott} that requires $s$ to be an odd power of $2$. The new argument works for all powers of $2$, so that we complete the proof of Theorem \ref{thm_STGQ} in Section 2.2 based on the work of Ott. In the case of finite quadrangles  with a point regular group $G$ of automorphisms, we use the Expander Crossing Lemma for bipartite graphs as found in \cite{large} to obtain a bound on $|\cL_2(g)|$ (see \eqref{eqn_reg_P2L2def} for definition) for $g\in G\setminus\{1\}$ in Section 3.1. By using this bound and the theory of solvable linear groups \cite{suprk}, we present the proof of Theorem \ref{thm_regGQ} in Section 3.2.

We use standard notation on group theory as is found in \cite{Gor}. Let $x,y$ be elements in a group $G$, and let $H,K$ be subgroups of $G$. We write $[x,y]=x^{-1}y^{-1}xy$, and set $[H,K]=\la [h,k]:\,h\in H,k\in K\ra$. We use $x^G$ for the conjugacy class of $G$ that contains $x$. We write $C_G(x)$ and $C_G(H)$ for the centralizers of $x$ and $H$ in $G$ respectively. We use $G'=[G,G]$ for the derived subgroup of $G$, and write $G''=[G',G']$ for the derived subgroup of $G'$.

\section{Skew translation generalized quadrangles}

Suppose that $\cS$ is an elation generalized quadrangle of order $(s,t)$ with base point $p$ and elation group $G$, and fix a point $y$ not collinear with $P$. Let $M_0,\ldots,M_t$ be the lines through $y$. By the generalized quadrangle property, there is exactly one point $z_i$ on each $M_i$ that is collinear with $P$. Set $A_i$, $A_i^*$ to be the stabilizers of $M_i$ and $z_i$ respectively for each $i$. Then each $A_i$ has order $s$, each $A_i^*$ has order $st$, and $A_i^*$ contains $A_i$. Moreover, we have
\begin{enumerate}
\item[(K1)] $A_iA_j\cap A_k=1$ for distinct $i,j,k$,
\item[(K2)] $A_i^*\cap A_j=1$ for distinct $i,j$.
\end{enumerate}
We call $(G,\{A_i\},\{A_i^*\})$ the $4$-gonal family associated with $\cS$. It is first noted in \cite{kantor} that we can reconstruct $\cS$ from the associated $4$-gonal family, see also \cite[p.~106]{FGQ}.

We now define a $4$-dimensional algebra from the $4$-gonal family $(G,\{A_i\},\{A_i^*\})$ associated with the elation quadrangle $\cS$ of order $(s,t)$. For a group $H$, we write $H^\#:=H\setminus\{1\}$. Let $\bR[G]$ be the group ring of $G$ over $\bR$, cf. \cite{GR}. For a subset $X$ of $G$, we also write $X$ for the sum in $\bR[G]$ of all the elements in $X$. We define two elements of $\bR[G]$ as follows:
\begin{align}
\Delta&:=A_0^\#+A_1^\#+\cdots+A_t^\#,\label{eqn_ST_Deltadef}\\
\Delta^*&:=(A_0^*)^\#+(A_1^*)^\#+\cdots +(A_t^*)^\#.\label{eqn_ST_Deltasdef}
\end{align}
\begin{lemma}\label{lem_4dim_Del}
Take notation as above. Then $\Delta$ and $\Delta^*$ commute, and we have
\begin{align*}
\Delta^2&=(s-2)(t+1)+(t+1)G+(s-t-2)\Delta-\Delta^*,\\
\Delta\Delta^*&=(s-t-1)(t+1)+t(t+1)G-(t+1)\Delta+(s-t-1)\Delta^*,\\
(\Delta^*)^2&=(st-t-1)(t+1)+t^2(t+1)G+(st-2t-2)\Delta^*.
\end{align*}
In particular, $\la 1,\Delta,\Delta^*,G\ra_{\bR}$ is a $4$-dimensional subalgebra of $\bR[G]$.
\end{lemma}
\begin{proof}
 For each $i$, the sets in $\{A_i^*\}\cup\{A_j^\#A_i:\,j\in\{0,\ldots,t\}\setminus\{i\}\}$ 
are pairwise disjoint by (K1) and (K2). They form a partition of $G$, since the sum of their sizes is $st+t(s-1)s=|G|$. We reformulate this fact as $(\Delta-A_i^\#)\cdot A_i+A_i^*=G$ in $\bR[G]$, i.e., $\Delta\cdot A_i^\#=G-A_i^*+(s-1)A_i-\Delta$. Taking summation over $i$, we deduce that
\begin{align*}
\Delta^2&=(t+1)G-(\Delta^*+t+1)+(s-1)(\Delta+t+1)-(t+1)\Delta\\
&=(s-2)(t+1)+(s-t-2)\Delta-\Delta^*+(t+1)G.
\end{align*}
This is the first equation in the lemma.

We have $A_jA_i^*=G$ for $j\ne i$ by (K2). Taking summation over $j\ne i$, we obtain $(\Delta+t-A_i^\#)A_i^*=tG$, i.e., $\Delta (A_i^*)^\#=(s-1)A_i^*+t(G-A_i^*)-\Delta=tG+(s-t-1)A_i^*-\Delta$. Taking summation over $i$, we obtain the second equation in the lemma. Starting with the fact $A_j^*A_i=G$ for $j\ne i$, we derive the same expression for $\Delta^*\Delta$, so $\Delta$ and $\Delta^*$ commute in $\bR[G]$. The third equation is derived in a similar way, and we omit the details. This completes the proof.
\end{proof}

\subsection{The characters $\chi_S,\chi_T$ of the elation group $G$}\label{subsec_egq}

For each element $g\in G$, we define two functions on $G$ as follows:
{\small
\begin{align}
\chi_S(g)&=\frac{1}{s(s+t)}
\sum_{i=0}^t\left(|C_G(g)|\cdot(s|A_i\cap g^G|+|A_i^*\cap g^G|)-|G/G'|\cdot(s|A_i\cap gG'|+|A_i^*\cap gG'|)\right), \label{eqn_chiS_Del}\\
\chi_T(g)&=\frac{\gcd(s,t)}{st(s+t)}\sum_{i=0}^t\left(|C_G(g)|\cdot(t|A_i\cap g^G|-|A_i^*\cap g^G|)-|G/G'|\cdot(t|A_i\cap gG'|-|A_i^*\cap gG'|)\right).\label{eqn_chiT_Del}
\end{align}}
\noindent The right hand sides of \eqref{eqn_chiS_Del} and \eqref{eqn_chiT_Del} only involve $|C_G(g)|$, $g^G$ and $gG'$, so $\chi_S$ and $\chi_T$ are constant on the conjugacy classes, i.e., they are class functions of $G$.  We establish the following result in this section.
\begin{thm}\label{thm_classfunc}
Suppose that $\cS$ is an elation generalized quadrangle of order $(s,t)$ with elation group $G$ and associated $4$-gonal family $(G,\{A_i\},\{A_i^*\})$. The class functions $\chi_S$, $\chi_T$  defined by \eqref{eqn_chiS_Del} and \eqref{eqn_chiT_Del} respectively are characters of $G$.
\end{thm}

The remaining part of this subsection is devoted to the proof of Theorem \ref{thm_classfunc}. For this purpose, we consider the character values of $\Delta,\Delta^*$ by examining the $4$-dimensional algebra in Lemma \ref{lem_4dim_Del}. For a complex character $\chi$ of $G$, its kernel is $\ker(\chi):=\{g\in G:\,\chi(g)=\chi(1)\}$. If $\chi$ is linear, then by \cite[Remark~4.2.11]{GR} we extend it to a ring homomorphism 
\[
  \chi:\,\bR[G]\rightarrow \mathbb{C},\quad \sum_{g\in G}a_gg\mapsto \sum_{g\in G}a_g\chi(g)\quad \textup{ for $a_g$'s in $\bR$}.
\] 

\begin{lemma}\label{lem_chiH}
If $\chi$ is a linear character of a finite group $G$ and $H$ is a subgroup, then $\chi(H)=|H|$ or $0$ according as $H$ is in $\ker(\chi)$ or not.
\end{lemma}
\begin{proof}
The claim is clear if $H$ is contained in $\ker(\chi)$, so assume that $H$ contains an element $g$ such that $\chi(g)\ne 1$. We deduce from $H=gH$ that $\chi(H)=\chi(g)\chi(H)$, i.e., $\chi(H)(1-\chi(g))=0$. It follows that $\chi(H)=0$ as desired. This completes the proof.
\end{proof}

\begin{proposition}\label{prop_linear_char}
Suppose that $\chi$ is a nonprincipal linear character of $G$, and let $u$ and $u'$ be the numbers of $A_i$ and $A_i^*$ that are contained in $\ker(\chi)$ respectively. Then we have
\begin{equation}\label{eqn_char_Del}
\chi(\Delta)=su-t-1, \quad\chi(\Delta^*)=stu'-t-1,
\end{equation}
where $(u,u')$ is one of $(1,1)$, $(0,0)$ and $(\frac{t}{s}+1,0)$.
\end{proposition}
\begin{proof}
 Since $\chi$ is a nonprincipal linear character, we deduce that $\chi(G)=0$ by Lemma \ref{lem_chiH}. For each $i$ we have  $\chi(A_i)=s$ or $0$ according as $A_i$ is contained in $\ker(\chi)$ or not by the same lemma, and there is a similar result for $A_i^*$. Taking summation over $i$, we obtain the expressions of $\chi(\Delta)$, $\chi(\Delta^*)$ as in \eqref{eqn_char_Del}, where $u,u'$ are as in the statement of the proposition.

Recall that  $A_iA_j^*=G$ for distinct $i,j$ by (K2) and $A_i<A_i^*$ for each $i$. It follows that $A_i^*A_j^*=G$ for $i\ne j$. We deduce that $\ker(\chi)$ contains at most one $A_i^*$, and if $A_i^*$ is in $\ker(\chi)$ then the only $A_j$ contained in $\ker(\chi)$ is $A_i$. In other words, we have $u'\le 1$, and $u'=1$ implies that $u=1$. We apply $\chi$ to both sides of the first equation in Lemma \ref{lem_4dim_Del} to obtain
\begin{align*}
(su-t-1)^2&=(s-2)(t+1)+(s-t-2)(su-t-1)-(stu'-t-1),
\end{align*}
which simplifies to $u^2s-ut-us+u't=0$. When $u'=0$, we deduce that $u=0$ or $u=\frac{t}{s}+1$. This completes the proof.
\end{proof}

\begin{corollary}\label{cor_s_notdiv_t}
Suppose that $s$ does not divide $t$.  Then $A_i^*G'=A_iG'$ for each $i$.
\end{corollary}
\begin{proof}
The claim is clear if $G=G'$, so assume that $G'<G$. We set $H:=G/G'$, and let $\phi:\bR[G]\rightarrow\bR[H]$ be the ring homomorphism that linearly extends the natural epimorphism from $G$ to $H$, cf. \cite[Corollary~3.2.8]{GR}. Set $X_1:=\phi(\Delta)+t+1$, $X_2=\phi(\Delta^*)+t+1$. Let $\chi$ be a linear character of $G$, which is naturally a character of $H$. If $\chi$ is principal, then $\chi(X_1)=s(t+1)$, $\chi(X_2)=st(t+1)$. If $\chi$ is nonprincipal, then $(\chi(X_1),\chi(X_2))$ is one of $(0,0)$, $(s,st)$ by Proposition \ref{prop_linear_char}. We thus have $\chi(X_2)=t\chi(X_1)$ for each character $\chi$ of $H$. It follows that $X_2=tX_1$ in $\bR[H]$. By examining the coefficient of the identity, we obtain $\sum_{i=0}^t(|A_i^*\cap G'|-t\cdot|A_i\cap G'|)=0$. On the other hand we deduce from $A_i^*G'\ge A_iG'$ that $\frac{|A_i^*|\cdot|G'|}{|A_i^*\cap G'|}\ge \frac{|A_i|\cdot|G'|}{|A_i\cap G'|}$, i.e., $t\cdot|A_i\cap G'|\ge |A_i^*\cap G'|$. It follows that $|A_i^*\cap G'|=t\cdot|A_i\cap G'|$, and so $A_i^*G'=A_iG'$ for each $i$. This completes the proof.
\end{proof}

We define the following elements of $\bR[G]$:
\begin{align}
S&:=s(\Delta+t+1)+(\Delta^*+t+1)=\sum_{i=0}^t\left( s A_i+A_i^*\right),\\
T&:=t(\Delta+t+1)-(\Delta^*+t+1)=\sum_{i=0}^t\left( t A_i-A_i^*\right).
\end{align}
Their connections with $\chi_S$, $\chi_T$ will become clear soon. Let $\chi$ be a linear character of $G$. If $\chi$ is principal, then $\chi(S)=s(s+t)(t+1)$, $\chi(T)=0$. If $\chi$ is nonprincipal, then $\chi(S)$ is one of $\{s(s+t),\,0\}$, and $\chi(T)$ is one of $\{t(s+t),\,0\}$ by Proposition \ref{prop_linear_char}.

Set $H:=G/G'$, and let $\hat{H}$ be its character group. We identify the characters of $H$ with the linear characters of $G$ in the natural way. Let $\phi:\bR[G]\rightarrow\bR[H]$ be the ring homomorphism that extends the natural epimorphism from $G$ to $H$. For each element $gG'$ of $H$, its coefficient in $\phi(\Delta+t+1)$ is $\sum_{i=0}^t|A_i\cap gG'|$ and its coefficient in $\phi(\Delta^*+t+1)$ is $\sum_{i=0}^t|A_i^*\cap gG'|$ respectively by the definitions of $\Delta$ and $\Delta^*$.  By applying the column orthogonality relation \cite[Theorem~2.8]{Gor}, we obtain
\begin{align}
\sum_{\chi\in\hat{H}}\chi(g)\overline{\chi(S)}&=|H|\cdot \sum_{i=0}^t\left(s|A_i\cap gG'|+|A_i^*\cap gG'|\right),\label{eqn_linchar_S}\\
\sum_{\chi\in\hat{H}}\chi(g)\overline{\chi(T)}&=|H|\cdot \sum_{i=0}^t\left(t|A_i\cap gG'|-|A_i^*\cap gG'|\right),\label{eqn_linchar_T}
\end{align}
for each $g\in G$. The right hand sides of \eqref{eqn_linchar_S}, \eqref{eqn_linchar_T} are divisible by $s(s+t)$ and $t(s+t)$ respectively, since the left hand sides are.

\begin{proposition}\label{prop_nonchi_ST}
Let $\chi$ be a nonlinear irreducible character of $G$. Then
\[
\chi(S)=s(s+t)\omega_\chi,\quad \chi(T)=\frac{(s+t)st}{\gcd(s,t)}z_\chi
\]
for some nonnegative integers $\omega_\chi,\,z_\chi$.
\end{proposition}
\begin{proof}
Let $\psi$ be a matrix representation for $\chi$, and write $m:=\chi(1)$. Let $I_{m}$ be the identity matrix of order $m$, so that $\psi(1)=I_m$.  For each $i$ we deduce that  $\psi(A_i)^2=s\psi(A_i)$ from $A_i^2=sA_i$. Since the polynomial $X^2-sX$ has two distinct roots $0,s$, $\psi(A_i)$ can be diagonalized and its eigenvalues are either $0$ or $s$. It follows that $\chi(A_i)=sc_{i,\chi}$ for some nonnegative integer $c_{i,\chi}$. Taking summation over $i$, we obtain  $\chi(\Delta)=su_\chi-(t+1)m$ with $u_\chi=\sum_ic_{i,\chi}\ge 0$. Similarly, $\chi(\Delta^*)=stv_\chi-(t+1)m$ with $v_\chi\ge 0$. It follows that $\chi(T)=st(u_\chi-v_\chi)$.

Since $\chi$ is irreducible, we have $\psi(G)=0$ by Schur's Lemma. Applying $\psi$ to the three equations in Lemma \ref{lem_4dim_Del}, we deduce that $\la 1,\psi(\Delta),\psi(\Delta^*)\ra$ is a $3$-dimensional commutative algebra, $\psi(\Delta^*)=g(\psi(\Delta))$ with $g(X)=-X^2+(s-t-2)X+(s-2)(t+1)$, and $\psi(\Delta)$ satisfies a polynomial $f(X)=(X-s+1)(X+t+1)(X+t+1-s)$ with three distinct roots. It follows that $\psi(\Delta)$ and $\psi(\Delta^*)$ can be simultaneously diagonalized. Let $x,y$ be eigenvalues of $\psi(\Delta)$ and $\psi(\Delta^*)$ with a common eigenvector $\mathbf{v}$. By  multiplying both sides of $\psi(\Delta^*)=g(\psi(\Delta))$ with $\mathbf{v}$, we deduce that $y=-x^2+(s-t-2)x+(s-2)(t+1)$. Since $x$ is a root of $f(X)=0$, it follows that
\[
(x,y)=(s-1,-t-1), \;(-t-1,-t-1), \;\textup{or}\; (s-t-1,st-t-1).
\]
The corresponding eigenvalues of $\psi(S)$ are $s(s+t),0,s(s+t)$ respectively in those cases, so $\chi(S)=s(s+t)\omega_\chi$ for some nonnegative integer $\omega_\chi$ by taking the trace of $\psi(S)$. The corresponding eigenvalues of $\psi(T)$ are $t(s+t),0,0$ respectively, so $\chi(T)$ is nonnegative and divisible by $t(s+t)$. Together with the fact that $st$ divides $\chi(T)$, we deduce that $\chi(T)$ is a nonnegative multiple of $\lcm(t(s+t),st)=\frac{st(s+t)}{\gcd(s,t)}$. This completes the proof.
\end{proof}

\begin{proof}[Proof of Theorem \ref{thm_classfunc}] Take notation as above. Let $\chi_1,\ldots,\chi_d$ be all the irreducible characters of $G$, where the first $r$ are linear and the last $d-r$ are nonlinear. By the column orthogonality relation \cite[Theorem~2.8]{Gor}, for each $g\in G$ it holds that
\begin{align}
\sum_{i=1}^d\chi_i(g)\overline{\chi_i(S)}&=|C_G(g)|\cdot\sum_{i=0}^t\left(s|A_i\cap g^G|+|A_i^*\cap g^G|\right), \label{eqn_allchar_S}\\
\sum_{i=1}^d\chi_i(g)\overline{\chi_i(T)}&=|C_G(g)|\cdot\sum_{i=0}^t\left(t|A_i\cap g^G|-|A_i^*\cap g^G|\right).\label{eqn_allchar_T}
\end{align}
For each nonlinear character $\chi_i$, let $\omega_{\chi_i}$ and $z_{\chi_i}$ be the nonnegative integers in Proposition \ref{prop_nonchi_ST}. From the equations  \eqref{eqn_linchar_S}-\eqref{eqn_allchar_T}  and Proposition \ref{prop_nonchi_ST}, we deduce that $\chi_S(g)=\sum_{i=r+1}^d\omega_{\chi_i}\chi_i(g)$ and $\chi_T(g)=\sum_{i=r+1}^dz_{\chi_i}\chi_i(g)$. We conclude that $\chi_S$ and $\chi_T$ are characters of $G$ as desired.
\end{proof}

\subsection{Proof of Theorem \ref{thm_STGQ}}

Take the same notation as in Section \ref{subsec_egq}, and assume that $\cS$ is a skew translation generalized quadrangle of even order $s$. By \cite{chen,Hachen}, $s$ is a power of $2$. By \cite[8.2.2]{FGQ}, $U_0:=\cap_{i=0}^tA_i^*$ is a normal subgroup of $G$ of order $s$, and $A_i^*=A_iU_0$ for each $i$. We deduce from the fact $A_i^*A_j^*=G$ that  $A_i^*\cap A_j^*=U_0$ for distinct $i,j$. The quotient images of the $A_i$'s form a spread of $G/U_0$, so $G/U_0$ is an elementary abelian $2$-group, cf. \cite[Proposition~2.4]{Hachen}. It follows that $\Phi(G),G'\le U_0$. For a property \textbf{P}, we define 
\[
    [\![\textbf{P}]\!]:=\begin{cases}1,\quad \textup{if \textbf{P} holds};\\0,\quad \textup{otherwise}.\end{cases}
\]
\begin{lemma}\label{lem_int_U0gAi}
Take notation as above, and assume that $0\le i,j\le t$.
\begin{enumerate}
\item[(1)] If $g\in U_0$, then $|A_i^*\cap gG'|=|G'|$ and $|A_i\cap gG'|=[\![g\in G']\!]$.
\item[(2)] If $g\in U_0$, then $|A_i^*\cap g^G|=|g^G|$ and $|A_i\cap g^G|=[\![g=1]\!]$.
\item[(3)] If $g\in A_j^*\setminus U_0$, then $|A_i^*\cap g^G|=|g^G|\cdot[\![i=j]\!]$.
\end{enumerate}
\end{lemma}
\begin{proof}
(1) Take $g\in U_0$. Since $G'\le U_0$ and $U_0\le A_i^*$, we have $gG'\subseteq A_i^*$ and so $|A_i^*\cap gG'|=|G'|$. If $A_i\cap gG'$ is not empty, then $g\in A_iG'$, and so $g\in A_iG'\cap U_0=G'(A_i\cap U_0)=G'$. Therefore, $A_i\cap gG'$ is nonempty only if $gG'=G'$, i.e., $g\in G'$. We have $|A_i\cap G'|=1$, since $A_i\cap U_0=1$ by (K2). This proves the claim.

(2) Take $g\in U_0$. Since $U_0$ is normal, $g^G$ is contained in $U_0$. The claim then follows from the facts $A_i\cap U_0=1$ and $U_0\le A_j^*$ for each $j$.

(3) Take $g\in A_j^*\setminus U_0$. The groups $U_0$ and $A_j^*$ are normal, so $g^G$ is contained in $A_j^*\setminus U_0$. It follows that $|A_j^*\cap g^G|=|g^G|$. Suppose that $i\ne j$. Since $A_i^*\cap A_j^*=U_0$ and $g^G\cap U_0=\emptyset$, we deduce that $|A_i^*\cap g^G|=0$. This completes the proof.
\end{proof}

\begin{corollary}
Let $\chi_S$, $\chi_T$ be the characters as defined in \eqref{eqn_chiS_Del}, \eqref{eqn_chiT_Del} respectively. Then they coincide on $U_0$, and for $g\in U_0$ we have
\begin{align}
\chi_S(g)&=\begin{cases}0, &\quad\textup{ if } g\in U_0\setminus G',\\
-\frac{s+1}{2s}\cdot|G/G'|,&\quad\textup{ if } g\in G'\setminus\{1\},\\
\frac{s+1}{2s}\cdot(s^3-|G/G'|),&\quad \textup{ if } g=1.
\end{cases}\label{eqn_skew_chiS}
\end{align}
\end{corollary}
\begin{proof}
Take $g\in U_0$. By \eqref{eqn_chiS_Del} and Lemma \ref{lem_int_U0gAi} we have
\begin{align*}
\chi_S(g)&=\frac{s+1}{2s^2}\cdot\left((s\cdot [\![g=1]\!]+|g^G|)\cdot |C_G(g)|-(s\cdot[[g\in G']]+|G'|)\cdot |G/G'|\right)\\
&=\frac{s+1}{2s}\cdot\left(|C_G(g)|\cdot[\![g=1]\!]-|G/G'|\cdot[\![g\in G']\!]\right),
\end{align*}
so \eqref{eqn_skew_chiS} follows. We obtain the same expression for $\chi_T(g)$ similarly.
\end{proof}

\begin{proof}[Proof of Theorem \ref{thm_STGQ}.] We recall the steps in the proof of Theorem \ref{thm_STGQ}  for the case $s$ is an odd power of $2$ in \cite{ott}: (1) show that $G'=\Phi(G)=[G,U_0]\Phi(U_0)<U_0$, cf. \cite[Theorem~2.7]{ott}, (2) show that $G'=[G,U_0]$, cf. \cite[Theorem~4.1]{ott}, (3) show that $G'=\Phi(U_0)$, cf. \cite[Theorem~4.11]{ott}, and (4) show that $G'=1$ and complete the proof, cf. \cite[Section~5]{ott}. It is explicitly stated on pages 403 and 413 of \cite{ott} that  the assumption that $s$ is not a square emerges first in \cite[Lemma~4.8]{ott} and only there.  Lemma 4.8 and the subsequent Lemmas 4.9-4.10 of \cite{ott} are used to accomplish step (2). The steps (3) and (4) are based on the conclusion in step (2) and do not invoke any of those three lemmas. We now present an alternative proof of step (2) which works for all powers of $2$, so that Theorem \ref{thm_STGQ} follows. For the sake of completeness, we first quote the relevant group theoretic results in \cite{ott}. Suppose to the contrary that $[G,U_0]<G'$. We have $G'=\Phi(G)$ by step (1). By the argument in the beginning of \cite[Section~4]{ott}, there is a normal subgroup $M$ of $G$ with the following properties:
\begin{equation}\label{eqn_skew_Mprop}
[G,U_0]\le M<U_0, \quad[\Phi(G):\,\Phi(G)\cap M]=2.
\end{equation}
We choose $M$ maximal with respect to these properties, and set $\bar{G}=G/M$. By \cite[Lemma~4.3]{ott}, $\Phi(\bar{G})=\overline{\Phi(G)}=\Phi(\overline{U_0})$ and they have order two, $\overline{U_0}\le Z(\bar{G})$, $\bar{G}$ is not abelian and $\overline{U_0}$ is not elementary abelian. By \cite[Lemma~2.6]{ott}, we have $Z(\bar{G})=\overline{U_0}$, so $Z(\bar{G})$ is not elementary abelian. By the classification of nonabelian $2$-groups admitting a Frattini subgroup of order $2$ \cite{extrasp}, there is an elementary abelian subgroup $E$, an extraspecial subgroup $S$ and a cyclic subgroup $C_4=\la w\ra$ of order $4$ such that $\bar{G}=E\times(S\circ C_4)$, where $\circ$ stands for central product. It follows that $\overline{U_0}=Z(\bar{G})=E\times C_4$. By \cite[Lemma~4.4]{ott}, $E\cap \overline{U_0}=1$, i.e., $E=1$ by the choice of $M$. As a corollary,  $\overline{U_0}=Z(\bar{G})=C_4$, $|M|=\frac{s}{4}$ and $\bar{G}=S\circ C_4$ has order $4s^2$, cf. \cite[Corollary~4.5]{ott}. The arguments so far work for all powers of $2$.

The extraspecial $2$-group $S$ has order $2s^2$, and its nonlinear irreducible characters all have degree $s$ by \cite[Theorem~5.5]{Gor}.  Write $z:=w^2$. Let $\chi$ be a nonlinear irreducible character of $\bar{G}$.  There is a linear character $\eta$ of $C_4$ of order $4$ and a nonlinear irreducible character $\chi_1$ of $S$ such that $\chi_1(z)=\eta(z)s$ and $\chi(xw^i)=\chi_1(x)\eta(w)^i$ for $x\in S$ and $i\ge 0$. The character $\chi_1$ of $S$ has degree $s$, and vanishes outside the center $Z(S)=\la z\ra$ due to the orthogonal relation \cite[Theorem~2.2]{Gor} and the fact $\chi_1(1)^2=[S:\,Z(S)]$. Therefore, $\chi(x)=0$ for $x\in \bar{G}\setminus C_4$, and $\chi(w^i)=\eta(w)^is$ for each $i$. We identify $\chi$ as a character of $G$ that contains $M$ in its kernel. Then  $\chi(g)=0$ for $g\in G\setminus U_0$, and $\chi(g)=\eta(\bar{g})s$ for $g\in U_0$, where $\bar{g}$ is the quotient image of $g$ in $\bar{G}$.

Write $M_1:=G'\cap M$. By \eqref{eqn_skew_Mprop} and the fact $\Phi(G)=G'$, we have $[G':\,M_1]=2$. We deduce from $\overline{U_0}=C_4$ and $\overline{G'}=\Phi(\overline{U_0})$ that $\overline{G'}=\la z\ra$. It follows that
\[
 \chi(g)=\begin{cases}s,&\quad \textup{ if } g\in M_1,\\-s,&\quad \textup{ if }g\in G'\setminus M_1.\end{cases}
\]
We now calculate $(\chi_S,\chi)_G$, cf. \cite[p.~120]{Gor}. Since $\chi$ vanishes outside $U_0$, we only need the value of $\chi_S(g)$ for $g\in U_0$, as is found in \eqref{eqn_skew_chiS}. Set $u:=|G/G'|$. We have
\begin{align*}
(\chi_S,\chi)_G
&=\frac{s+1}{2 s^4}\left((s^3-u)\cdot s+(-u)\cdot s\cdot(|M_1|-1)+(-u)\cdot (-s)\cdot(|G'|-|M_1|)\right)\\
&=\frac{1}{2}(s+1)+\frac{(s+1)u}{2s^3}\cdot\left(-1-|M_1|+1+|G'|-|M_1|\right)\\
&=\frac{1}{2}(s+1).
\end{align*}
This is not an integer, contradicting the fact that $\chi_S$ is a character of $G$. This completes the proof of step (2) for all powers of $2$ and so of Theorem \ref{thm_STGQ}.
\end{proof}

\section{Generalized quadrangles with a point regular group of automorphisms}

Suppose that $\mathcal{S}=(\cP,\cL)$ is a finite thick generalized quadrangle of order $(s,t)$ with a group $G$ of automorphisms acting regularly on the points. Here, $\cP$ is the set of points and $\cL$ is the set of lines. We have $|G|=|\cP|=(1+s)(1+st)$. If two distinct points $P,Q$ are collinear, we write $P\sim Q$. Similarly, if two distinct lines $\ell,m$ are concurrent, we write $\ell\sim m$. We fix a point $O$, and identify $O^g$ with $g$ for each $g\in G$. Set
\begin{equation}\label{eqn_Delta_def}
\Delta:=\{g\in G:\,O^g\sim O\},
\end{equation}
which has size $(t+1)s$. If $O\sim O^g$ for some $g\in G$, then clearly $O\sim O^{g^{-1}}$. It follows that $\Delta=\Delta^{(-1)}$, where $\Delta^{(-1)}=\{g^{-1}:\,g\in\Delta\}$.
\begin{lemma}\label{lem_Delta2}
We have $\Delta^2=(t+1)(s-1)+(s-t-2)\Delta+(t+1)G$ in $\bR[G]$.
\end{lemma}
\begin{proof}
Take $g\in G\setminus\{1\}$. The neighbors of $O^g$ are $O^{xg}$, $x\in\Delta$. If $g\in\Delta$, then $O$ and $O^g$ are collinear and there are $s-1$ points that are collinear with both of them. If $g\not\in\Delta$, then $O$ and $O^g$ are not collinear. In this case, there is exactly one point on each line through $O$ that is collinear with $O^g$. In other words, we have
\[
|\Delta\cap \Delta g|=\begin{cases}s-1,&\quad\textup{ if }g\in\Delta,\\t+1,&\quad\textup{ if } g\in G\setminus (\Delta\cup\{1\}).\end{cases}
\]
Recall that $|\Delta|=(t+1)s$ and $\Delta=\Delta^{(-1)}$. It follows that $|\Delta\cap \Delta g|$ is the coefficient of $g$ in $\Delta^2$ for each $g\in G$. Therefore, we have
\begin{align*}
\Delta^2&=(t+1)s+(s-1)\Delta+(t+1)(G-1-\Delta)\\
        &=(t+1)(s-1)+(s-t-2)\Delta+(t+1)G.
\end{align*}
This completes the proof.
\end{proof}

\begin{remark}
Our definition of $\Delta$ is slightly different from that in \cite{Greg,Yreg}, since we exclude $1$ from $\Delta$. In the case $s=t$, Lemma \ref{lem_Delta2} is equivalent to the fact that $\Delta\cup\{1\}$ is a $((1+s)(1+s^2),s^2+s+1,s+1)$ difference set with multiplier $-1$ in $G$, cf. \cite{Greg}.
\end{remark}

Fix an element $g\in G\setminus\{1\}$. Let $\cL_1(g)$ be the set of lines stabilized by $g$, and define
\begin{align}\label{eqn_reg_P2L2def}
\cP_2(g):=\{P\in\cP:\,P\sim P^g\},\quad \cL_2(g):=\{\ell\in\cL:\,\ell^g\sim \ell\}.
\end{align}
The following result \cite[Lemma~3]{Yreg}, which is a corollary of  \cite[1.9.1,~1.9.2]{FGQ}, is a standard tool in the analysis of generalized quadrangles with point regular groups of automorphisms.
\begin{lemma}\label{lem_Y_Benson}
For each $g\in G\setminus\{1\}$, we have
\begin{align*}
|\cP_2(g)|&=|C_G(g)|\cdot|g^G\cap\Delta|=(s+1)\cdot|\cL_1(g)|+|\cL_2(g)|\\
          &\equiv (s+1)(t+1)\pmod{s+t},
\end{align*}
and $|C_G(g)|\cdot|g^G\cap\Delta^c|\equiv t(s^2-1)\pmod{s+t}$,
where $\Delta^c$ is the complement of $\Delta$ in $G$.
\end{lemma}

\subsection{An upper bound on $\cL_2(g)$}

In this subsection, we establish the following upper bound on $\cL_2(g)$, which will play a crucial role in the proof of Theorem \ref{thm_regGQ}.
\begin{thm}\label{thm_l2g_upp}
Suppose that $\cS$ is a thick generalized quadrangle of order $(s,t)$ with a point regular group $G$ of automorphisms. Then for each $g\in G$ such that $g^2\ne 1$, we have $|\cL_2(g)|< (1+st)(2+\sqrt{s+t})$, where $\cL_2(g)$ is as defined in \eqref{eqn_reg_P2L2def}.
\end{thm}

We shall need the following version of Expander Crossing Lemma for bipartite graphs as found in \cite[Lemma~8]{large}, see also \cite[Theorem~5.1]{haemers}.
\begin{lemma}\label{lem_expander_e}
Let $G$ be a bipartite graph with parts $U$ and $V$ such that each vertex of $U$ has degree $a$ and each vertex of $V$ has degree $b$, and $X \subset U$ and $Y \subset V$. Let $A$ be the adjacency matrix of $G$ with eigenvalues $\lambda_1,\lambda_2,\ldots$ such that $|\lambda_1|\ge|\lambda_2|\ge\cdots$. Let $e(X,Y)$ be the number of edges between vertices in $X$ and $Y$. Then
\[
\left|e(X,Y)-\frac{\sqrt{ab}}{\sqrt{|U|\cdot|V|}}\cdot|X|\cdot|Y|\right|
\le |\lambda_3|\cdot \sqrt{|X|\cdot|Y|\cdot\left(1-\frac{|X|}{|U|}\right)\cdot\left(1-\frac{|Y|}{|V|}\right)}.
\]
\end{lemma}

The remaining part of this subsection is devoted to the proof of Theorem \ref{thm_l2g_upp}. For an element $g\in G$ such that $g^2\ne 1$, we define
\[
  \cP_2'(g):=\{P\in\cP_2(g):\,P,P^g,P^{g^{-1}}\textup{ are not collinear}\}.
\]
\begin{lemma}\label{lem_cP2_2lines}
Take an element $g\in G$ and a point $P\in\cP_2(g)$, and suppose that $g^2\ne 1$. Then $\ell_1:=PP^{g^{-1}}$, $\ell_2:=PP^g$ are all the lines of $\cL_1(g)\cup\cL_2(g)$ that pass through $P$. Moreover,
\begin{enumerate}
\item[(1)] if $P\not\in\cP_2'(g)$, then $\ell_1=\ell_2$ and lie in $\cL_1(g)$;
\item[(2)] if $P\in\cP_2'(g)$, then $\ell_1$ and $\ell_2$ are distinct and both lie in $\cL_2(g)$.
\end{enumerate}
\end{lemma}
\begin{proof}
We observe that $P,P^g,P^{g^{-1}}$ are three distinct points, $P$ is collinear with $P^g$, and $\ell_2=\ell_1^g$. We consider two separate cases.
\begin{enumerate}
\item[(i)] Suppose that $\ell$ is a line of $\cL_1(g)$ that passes through $P$. Since $\ell$ is $g$-invariant and $P$ is on $\ell$, we deduce that $P,P^g,P^{g^{-1}}$ are all on $\ell$. This implies that $P\not\in \cP_2'(g)$ and $\ell=\ell_1=\ell_2$. In particular, we see that $\ell_1$ is in $\cL_1(g)$.
\item[(ii)] Suppose that $\ell$ is a line in $\cL_2(g)$ that passes through $P$. Write $Q:=\ell\cap\ell^g$. The point $P^g$ is incident with $\ell^g$ and collinear with $P$. The three points $P,Q,P^g$ are not pairwise distinct, since otherwise they would form a triangle. If $P=Q$, then  $P$ is on both $\ell$ and $\ell^g$, which implies that $\ell$ is incident with $P$ and $P^{g^{-1}}$, i.e., $\ell=\ell_1$. If $P^g=Q$, then similarly we deduce that $\ell=\ell_2$.
\end{enumerate}
This proves the first claim. If $P\not\in \cP_2'(g)$, then $P,P^g,P^{g^{-1}}$ are collinear and so $\ell_1=\ell_2$, i.e., $\ell_1$ is in $\cL_1(g)$. If $P\in\cP_2'(g)$, then $\ell_1$ and $\ell_2$ are distinct and intersect at $P$. It follows that $\ell_1$ is in $\cL_2(g)$. By applying $g$, we see that $\ell_2$ and $\ell_2^g$ are distinct and intersect at $P^g$, so $\ell_2$ is in $\cL_2(g)$. This completes the proof.
\end{proof}

\begin{lemma}\label{lem_cL2_2pts}
Take an element $g\in G$ and a line $\ell\in\cL_2(g)$, and suppose that $g^2\ne 1$. There are exactly two points of $\cP_2(g)$ incident with $\ell$, and they both lie in $\cP_2'(g)$.
\end{lemma}
\begin{proof}
The line $\ell$ and $\ell^g$ are distinct, and we set $Q:=\ell\cap\ell^g$. We observe that $Q,Q^g,Q^{g^{-1}}$ are three distinct points, and $\ell=QQ^{g^{-1}}$, $\ell^g=QQ^g$. It follows that  $Q^g$, $Q$, $Q^{g^{-1}}$ are not collinear, so $Q$ is in $\cP_2'(g)$. Similarly, we deduce that $Q^{g^{-1}}$ is also in $\cP_2'(g)$.

Suppose that $P$ is a point  of $\cP_2(g)$ that is incident with $\ell$. The three points $P,Q,P^g$ can not be pairwise distinct, since otherwise they would form a triangle. It follows that $P$ is one of $Q,Q^{g^{-1}}$.  This completes the proof.
\end{proof}

\begin{proof}[Proof of Theorem \ref{thm_l2g_upp}] Take an element $g\in G$ such that $g^2\ne 1$. Let $\Gamma(\mathcal{S})$ be the point-line incidence graph of the quadrangle $\cS=(\cP,\cL)$. By Lemmas \ref{lem_cP2_2lines} and \ref{lem_cL2_2pts}, the induced subgraph of $\Gamma(\mathcal{S})$ on $\cP_2'(g)\cup\cL_2(g)$ is the union of disjoint cycles. It follows that $|\cP_2'(g)|=|\cL_2(g)|$, and there are $2\cdot|\cL_2(g)|$ edges between $\cP_2'(g)$ and $\cL_2(g)$. We apply Lemma \ref{lem_expander_e} to the graph $\Gamma(\mathcal{S})$ with $U=\cP$, $V=\cL$, $X=\cP_2'(g)$, $Y=\cL_2(g)$. The spectrum of the point graph of $\cS$ is available in the proof of \cite[1.2.2]{FGQ}, and the information about the $\lambda_i$'s can be deduced from it by linear algebra. We list the parameters as follows:  $|U|=(1+s)(1+st)$, $|V|=(1+t)(1+st)$, $a=t+1$, $b=s+1$, $\lambda_1=-\lambda_2=(t+1)(s+1)$, and $|\lambda_3|=\sqrt{s+t}$. If $|\cL_2(g)|=0$, the desired bound is trivial, so assume that $|\cL_2(g)|>0$. Since $|X|=|Y|=|\cL_2(g)|$ and $e(X,Y)=2\cdot|\cL_2(g)|$, we deduce that
\[
\left|2\cdot|\cL_2(g)|-\frac{|\cL_2(g)|^2}{1+st}\right|
<\sqrt{s+t}\cdot |\cL_2(g)|.
\]
Here we used the fact that $(1-\frac{|X|}{|U|})\cdot(1-\frac{|Y|}{|V|})<1$.   If $|\cL_2(g)|>0$, then it simplifies to $\left|2-\frac{|\cL_2(g)|}{1+st}\right|< \sqrt{s+t}$, so $|\cL_2(g)|< (1+st)(2+\sqrt{s+t})$. This completes the proof.
\end{proof}

\subsection{Proof of Theorem \ref{thm_regGQ}}

This whole section is devoted to the proof of Theorem \ref{thm_regGQ}. Suppose that $\cS$ is a generalized quadrangle of even order $s$ with a group $G$ of automorphisms acting regularly on points. By \cite[5.2.3]{FGQ}, the symplectic quadrangle $W(2)$ is the only generalized quadrangle of order $2$, and it does not admit a point regular group of automorphisms. Therefore, we assume that $s\ge 4$ in the sequel.  Since $s$ is even, $1+s$ and $1+s^2$ are  relatively prime odd integers. The group $G$ has order $(1+s)(1+s^2)$, so it is solvable.

\begin{lemma}\label{lem_Gbarorder_seqt}
The order of $G/G'$ divides either  $1+s$ or $1+s^2$.
\end{lemma}
\begin{proof}
Let $\chi$ be a nonprincipal linear character of $G$. Then $\chi(G)=0$. We apply $\chi$ to the equation in Lemma \ref{lem_Delta2} to obtain $\chi(\Delta)^2=s^2-1-2\chi(\Delta)$. It follows that $\chi(\Delta)$ is either $s-1$ or $-s-1$. Now suppose that $\chi$ has prime order $r$.  We claim that $\chi(\Delta)=-s-1$ or $s-1$ according as $r$ divides $1+s$ or $1+s^2$. Let $\omega$ be a complex primitive $r$-th root of unity,  $R$ be the algebraic integer ring of $\bQ(\omega)$, and $\mathfrak{P}$ be the prime ideal $(\omega-1)R$ lying over $r$. Since $|\Delta|=s(1+s)$, $\chi(\Delta)$ is the sum of $s^2+s$ complex $r$-th roots of unity. We deduce that $\chi(\Delta)\equiv s^2+s\pmod{\mathfrak{P}}$, from which the claim follows.

Suppose that there are odd prime divisors $r_1,r_2$  of $|G/G'|$ such that $r_1$ divides $1+s$ and $r_2$ divides $1+s^2$. Then $G$ has a linear character $\chi$ of order $r_1r_2$. We have $\chi(\Delta)=\epsilon s-1$ for some $\epsilon=\pm 1$, and $\chi^{r_1}(\Delta)=s-1$,  $\chi^{r_2}(\Delta)=-s-1$ by the previous paragraph. Let $\omega$ be a complex primitive $r_1r_2$-nd root of unity,  $R$ be the algebraic integer ring of $\bQ(\omega)$, and $\mathfrak{P}_i$ be a prime ideal of $R$ lying over $r_i$ for $i=1,2$. We have $\chi(\Delta)^{r_i}\equiv \chi^{r_i}(\Delta)\pmod{\mathfrak{P}_i}$ for each $i$, so $(\epsilon s-1)^{r_1}\equiv s-1\pmod{r_1}$,  $(\epsilon s-1)^{r_2}\equiv -s-1\pmod{r_2}$. Since $s$ is relatively prime to $|G|$, we derive the contradiction that $\epsilon=1$ and $\epsilon=-1$ both hold.  This completes the proof.
\end{proof}

\begin{proposition}\label{prop_gGcapDel}
Take notation as above. For $g\in G\setminus\{1\}$, we have $|g^G\cap\Delta|\ge 1$, $|g^G\cap\Delta^c|\equiv s\pmod{2s}$. In particular, $|g^G\cap\Delta|\equiv |g^G|\pmod{s}$ and $|g^G|\ge 1+s$.
\end{proposition}
\begin{proof}
By Lemma \ref{lem_Y_Benson}, we have $|g^G\cap\Delta|\cdot|C_G(g)|\equiv 1\pmod{2s}$. It follows that $|g^G\cap\Delta|\ne 0$. We multiply both sides by $(1+s)\cdot|g^G|$ to deduce that $|g^G\cap\Delta|\equiv |g^G|(1+s)\pmod{2s}$. Since $|g^G|$ divides $|G|$ which is odd, we deduce that  $|g^G\cap\Delta^c|\equiv s\pmod{2s}$. In particular, $|g^G\cap\Delta^c|$ is at least $s$. This completes the proof.
\end{proof}

\begin{lemma}\label{lem_Greg_coll}
Take notation as above. Then the following hold.
\begin{enumerate}
\item[(1)] The Fitting subgroup $F(G)$ is a $p$-group with $p$ a prime, $|F(G)|\ge s+1$.
\item[(2)] If $N$ is an elementary abelian normal $p$-subgroup of $G$, then $|N|$ divides $1+s^2$.
\item[(3)] Each nontrivial normal subgroup $M$ of $G$ has order at least $2s+3$.
\end{enumerate}
Moreover, we have $p\equiv 1\pmod{4}$, where $p$ is as in (1).
\end{lemma}
\begin{proof}
The first two claims are respectively \cite[Theorems~4.2,~3.5]{Greg} for even $s$. For (3), we take an element $g\in M\setminus\{1\}$. By the same argument as in the proof of \cite[Lemma~3.6]{Swreg}, $g$ and $g^{-1}$ are not conjugate in $G$. It follows that $|M|\ge 1+|g^G|+|(g^{-1})^G|\ge 2s+3$ by Proposition \ref{prop_gGcapDel}. Finally, let $p$ be as in the statement of (1). The elements of order $p$ in the center of $F(G)$ generate a nontrivial elementary abelian normal $p$-subgroup $N_1$ of $G$, so $p$ divides $1+s^2$. Since $s^2\equiv -1\pmod{p}$, $-1$ is a quadratic residue modulo $p$ and so $p\equiv 1\pmod{4}$. This completes the proof.
\end{proof}

The Fitting subgroup $F(G)$ is a $p$-group with $p\equiv 1\pmod{4}$ by Lemma \ref{lem_Greg_coll}. Among all nontrivial elementary abelian normal subgroups of $G$, we choose one of the smallest order, say, $N$. By the choice of $N$, it is a minimal normal subgroup of $G$.  It is nilpotent and thus contained in $F(G)$, so $|N|=p^d$ for an integer $d$. By Lemma \ref{lem_Greg_coll}, $p^d$ divides $1+s^2$. Set
\begin{equation}\label{eqn_Hdef}
   H:=G/C_G(N).
\end{equation}
The group $G$ act on $N$ via conjugation, and this induces a natural faithful action of $H$ on $N$. We regard $N$ as a vector space of dimension $d$ over $\F_p$, so that $H$ embeds as a subgroup of $\GL_d(p)$. Under such an embedding, the $H$-orbit of $g\in N$ corresponds to the conjugacy class $g^G$, and an $H$-submodule of $N$ corresponds to a normal subgroup of $G$ that is contained in $N$. Since $N$ is minimal normal in $G$, we deduce that $H$ is irreducible on $N$. We summarize the properties of $N$ and $H$ in the next lemma.

\begin{lemma}\label{lem_Hprop}
Take notation as above. The group $H$ is irreducible on $N$, and each $H$-orbit on $N\setminus\{1\}$ has size at least $1+s$. Moreover, we have $|N|\ge 2s+3$, $|H|\ge 1+s>|N|^{1/2}$, and $|C_G(N)|\le 1+s^2$.
\end{lemma}
\begin{proof}
By Lemma  \ref{lem_Greg_coll} and the arguments preceding this lemma, it remains to establish the last two inequalities.  Since $|H|$ is at least as large as its orbit sizes, we deduce that $|H|\ge 1+s$. It follows that $|C_G(N)|=\frac{|G|}{|H|}\le 1+s^2$. Since $|N|$ divides $1+s^2$, we have $|N|^{1/2}\le\sqrt{1+s^2}<1+s$. This completes the proof.
\end{proof}

\begin{lemma}\label{lem_N_minn}
The group $N$ is the unique minimal normal subgroup of $G$, and it is contained in each nontrivial normal subgroup of $G$.
\end{lemma}
\begin{proof}
It suffices to prove the second claim, since the first claim follows from it.  Suppose that $M$ is a nontrivial normal subgroup of $G$ that does not contain $N$. Then $M\cap N=1$ by the minimality of $N$. It follows that $[M,N]\le M\cap N=1$, i.e., $M\le C_G(N)$. The group $C_G(N)$ contains $MN$ which has order  $|N|\cdot |M|\ge (2s+3)^2>1+s^2$ by Lemma \ref{lem_Greg_coll} (3). It contradicts the fact $|C_G(N)|\le 1+s^2$ in Lemma \ref{lem_Hprop}. This completes the proof.
\end{proof}

\begin{corollary}\label{cor_1dim_FinK}
The group $F(G)$ is contained in  $C_G(N)$.
\end{corollary}
\begin{proof}
Since $F(G)$ is a $p$-group, its center $Z(F(G))$ is a nontrivial abelian normal subgroup of $G$. By Lemma \ref{lem_N_minn}, $Z(F(G))$ contains $N$. The claim then follows.
\end{proof}

\begin{lemma}\label{lem_1dim_abelian}
Each abelian normal subgroup of $G$ is an elementary abelian $p$-group.
\end{lemma}
\begin{proof}
Let  $L$ be a nontrivial abelian normal subgroup of $G$. We have $N\le L$ by Lemma \ref{lem_N_minn}, so $L\le C_G(N)$. If $L$ is not a $p$-group, then $O_{p'}(L)$ is a  nontrivial normal subgroup of $G$ which does not contain $N$. This contradicts Lemma \ref{lem_N_minn}, so $L$ is a $p$-group.

It remains to show that $L$ is elementary abelian. Suppose to the contrary that $L$ has exponent $p^e$ with $e\ge 2$. We have $L=C_{p^{e}}^{r_e}\times\cdots\times C_{p}^{r_1}$ for some nonnegative integers $r_1,\ldots,r_e$, where $r_e>0$ and $C_{p^i}$ is a cyclic subgroup of order $p^i$ for each $i$. Then $L_1:=\{g^{p^{e-1}}:\,g\in L\}$ is a characteristic subgroup of $L$ of order $p^{r_e}$. Since it is normal in $G$, it contains $N$ by Lemma \ref{lem_N_minn}. It follows that $|L|\ge |L_1|^e\ge |N|^2\ge (2s+3)^2$, which contradicts the fact $|C_G(N)|\le 1+s^2$ in Lemma \ref{lem_Hprop}. This completes the proof.
\end{proof}

Let $H$ be as in \eqref{eqn_Hdef}, so that $N$ is an irreducible $H$-module. We say that $N$ is an imprimitive $H$-module if $H$ preserves a decomposition
\begin{equation}\label{eqn_impr_D}
\cD:\,N=N_1\oplus \cdots\oplus N_t, \quad t\ge 2,
\end{equation}
where each $N_i$ is an $\F_p$-subspace of the same size $p^m$. That is, each element of $H$ induces a permutation of the set $\{N_1,\ldots,N_t\}$. Otherwise, we say that $N$ is primitive.
\begin{lemma}\label{lem_reg_impr}
The $H$-module $N$ is primitive.
\end{lemma}
\begin{proof}
Suppose to the contrary that $H$ preserves a decomposition $\cD$ as in \eqref{eqn_impr_D}. Since $H$ is irreducible, its induced action on $\{N_1,\ldots,N_t\}$ is transitive. It follows that $t$ is odd and so $t\ge 3$ by the fact that $|H|$ is odd. The group $H$ stabilizes the set $X=\{(x_1,\ldots,x_t):\,x_i\in N_i,\,\textup{all but one are the identities}\}$ with $|X|=(p^m-1)t$. It follows that $X$ is the union of conjugacy classes of $G$, and so $|X|\ge 1+s>|N|^{1/2}$ by Lemma \ref{lem_Hprop}. It follows that $(p^m-1)t>p^{mt/2}$, which holds only if $p^{m(t-2)}<t^2$. Since $t$ is odd and $p\equiv 1\pmod{4}$, the latter is satisfied only when $t=3$ and $p^m=5$. In this case, the stabilizer $M$ of $\cD$ in $\GL_3(5)$ is $\F_5^*\wr S_3$. It follows that $|H|=3$ since it divides the odd part of $|M|=3\cdot 2^7$. We have $|N|=5^3>|H|^2$, which contradicts Lemma \ref{lem_Hprop}. This completes the proof.
\end{proof}

\begin{lemma}\label{lem_prim_1dim}
The group $H$ lies in $\GamL_1(p^d)$.
\end{lemma}
\begin{proof}
Suppose to the contrary that $H$ does not lie in $\GamL_1(p^d)$. We take the same notation as in \cite{suprk} in this proof. Let $M$ be a maximal solvable subgroup of $\GL_d(p)$ that contains $H$.  By \cite[Theorem~19.5]{suprk} and Lemma \ref{lem_reg_impr}, $N$ is a primitive $M$-module. By \cite[Lemma~19.1,~Theorem~20.9]{suprk}, $M$ has a unique maximal abelian normal subgroup $F$ which is the multiplicative group of an extension field $K$ of $\F_p$, and $b:=[K:\,\F_p]$ divides $d$. Set $r:=\frac{d}{b}$, $q:=p^b$. Since $N$ is a faithful $F$-module, we regard $N$ as a $r$-dimensional vector space $V$ over $K$. In this way, $M$ embeds as a subgroup of $\GamL(V)$. Set $M_1:=M\cap\GL(V)$. We have $r>1$ by the assumption, so $M_1\ne F$ by \cite[Theorem~20.2]{suprk}. By \cite[Theorem~20.11]{suprk}, there is a unique subgroup $A$ of $M$ with the following properties: (i) $A$ contains $F$ in its center, (ii) $A/F$ is an abelian normal subgroup of $M/F$,  (iii) $A$ is maximal among the subgroups of $M$ satisfying (i) and (ii). In particular, $A$ is normal in $M$. By \cite[Theorem~20.3]{suprk}, we have $|A/F|=r^2$.

Suppose that $r=p_1^{\ell_1}\cdots p_k^{\ell_k}$ is  the prime decomposition of $r$. Let $Q_i$ be the full preimage of the unique Sylow $p_i$-subgroup of $A/F$, which is normal in $M$. The $Q_i$'s each have center $F$ and pairwise commute by \cite[Corollary~20.5.1,~Lemma~20.4]{suprk}, so $A$ is the central product of the $Q_i$'s. Correspondingly, $V=V_1\otimes \cdots\otimes V_k$, where each $V_i$ is an absolutely irreducible $Q_i$-module of degree $p_i^{\ell_i}$ over $K$, cf. \cite[Theorem~20.6]{suprk}. The normalizer of $A$ in $\GL(V)$ is the central product of $N_{\GL(V_i)}(Q_i)$'s by \cite[Theorem~20.13]{suprk}, and it contains $M_1$.

We claim that $k=1$. Suppose to the contrary that $k\ge 2$. The set $X$ of nonzero pure tensors $v_1\otimes\cdots\otimes v_k$ with each $v_i\in V_i$ is $M$-invariant  and thus $H$-invariant. We have
\[
|X|=(q-1)\prod_{i=1}^k\left(\frac{q^{r_i}-1}{q-1}\right),\quad\textup{ where $r_i:=p_i^{\ell_i}$ for each $i$.}
\]
By Lemma \ref{lem_Hprop}, $|X|>|V|^{1/2}=q^{r/2}$. This holds only if $2^k\cdot q^{r_1+\cdots +r_k+1-k}>q^{r/2}$, since $q-1>\frac{q}{2}$. Taking logarithm with base $q$, we deduce that $r<\sum_i 2r_i+2-2k+2k\log_q(2)$. Since $q\ge 5$ and $k\ge 2$, we have $2-2k+2k\log_q(2)\le -2+4\log_5(2)<0$. It follows that $r<2\sum_i r_i$. Set $x_i:=r_i-2$, which is nonnegative, for each $i$. We have $r=\prod_{i}(x_i+2)\ge 2^{k-1}\sum_{i}x_i+2^k$. Since $2\sum_i r_i=2\sum_{i}x_i+2k$ and $k\ge 2$, we derive the contradiction that $r\ge 2\sum_i r_i$.  This establishes the claim that $k=1$.

We now have $A=Q_1$, $r=p_1^{\ell_1}$, $q=p^b$. The group $A/F$ is an elementary abelian $p_1$-group of order $r^2$, and $p_1$ divides $q-1$ by \cite[Theorem~20.5]{suprk} and its corollary. The group $M_1/A$ embeds in $\Sp_{2\ell_1}(p_1)$ by \cite[Theorem~20.15]{suprk}. The order of $M$ divides
\begin{equation}\label{eqn_prim_z}
z:=|\Sp_{2\ell_1}(p_1)|\cdot p_1^{2\ell_1}\cdot (q-1)\cdot b.
\end{equation}
The order of $H$  divides the odd part of $z$, so the latter is larger than $|V|^{1/2}$ by Lemma \ref{lem_Hprop}. The maximum of $\log_{p^b}(b)$ is attained when $(p,b)=(5,3)$, which is smaller than $0.23$.  We take logarithm with base $q$ to deduce that
\begin{align}\label{eqn_C6_rm}
\log_q\left((|\textup{Sp}_{2\ell_1}(p_1)|\cdot p_1^{2\ell_1})_{2'}\right)+1.23>0.5p_1^{\ell_1}.
\end{align}
where $x_{2'}$ is the largest odd divisor of a nonzero integer $x$. Since $q\ge 5$ by Lemma \ref{lem_Greg_coll}, we replace $q$ by $5$ in the base of the logarithms on the left to obtain an inequality that only involves $p_1,\,\ell_1$. The new inequality is satisfied by exactly eleven $(p_1,\ell_1)$ pairs such that $p_1\equiv 1\pmod{5}$,  and each pair satisfies that $p_1\le 13$ and $\ell_1\le 5$.  For the cases where $(p_1,\ell_1)\ne(2,1)$,  the quantity $y:=(0.5p_1^{\ell_1}-1.23)^{-1}$ is positive and $q$ is upper bounded by $D:=\lfloor (|\textup{Sp}_{2\ell_1}(p_1)|\cdot p_1^{2\ell_1})_{2'}^{y}\rfloor$ by \eqref{eqn_C6_rm}. For each such $q=p^b$ that $p\equiv 1\pmod{4}$, $q\equiv 1\pmod{p_1}$ and $q\le D$, there is a unique even integer $s$ such that $0\le s<q^r-1$ and $q^r$ divides $1+s^2$, where $r=p_1^{\ell_1}$.  We then calculate $|G|=(1+s)(1+s^2)$ and the value $z$ in \eqref{eqn_prim_z}. The order of $H$ divides $\gcd(|G|,z)$, which turns out to be smaller than $1+s$ in all cases.  This contradicts the condition $|H|\ge 1+s$ in Lemma \ref{lem_Hprop}. This excludes all the $(p_1,\ell_1)$ pairs except for $(2,1)$.

It remains to consider the case $(p_1,\ell_1)=(2,1)$. In this case, $r=2$, $K=\F_{q}$ and $|V|=q^2$. The group $M_1$ has order $24(q-1)$, and is generated by $F$ and
\[
  \begin{pmatrix}1&0\\0&-1\end{pmatrix},\;
  \begin{pmatrix}0&1\\1&0\end{pmatrix},\;
  \begin{pmatrix}1&0\\0&-i\end{pmatrix},\;
  \begin{pmatrix}1&i\\1&-i\end{pmatrix}
\]
by \cite[Theorem~21.6]{suprk}. Here, $i$ is a square root of $-1$ in $\F_p$. The group $M$ is generated by $M_1$ and the semilinear map $\sigma:\,(v_1,v_2)\mapsto(v_1^p,v_2^p)$ for $(v_1,v_2)\in V$. Set $X:=\{(\lambda a,\lambda b):\,a,b\in\F_p,\lambda\in \F_q\}$, and let $Y$ be the set of $1$-dimensional $\F_q$-subspaces contained in $X$. They are both preserved by $M$, and $|Y|=p+1$. Let $U$ be the kernel of the action of $M$ on $Y$, which contains $F$ and $\sigma$. The group $\overline{M}:=M/U$ has order $24$, and its subgroup $\bar{H}$ thus has order $1$ or $3$. Therefore, the $\bar{H}$-orbits of $Y$ have sizes $1$ or $3$, and so $H$ stabilizes a subset $S$ of $X$ of size $q-1$ or $3(q-1)$. Since $|H|$ is odd, there is an $H$-orbit of  size at most $\frac{3}{2}(q-1)$ contained in $S$. Since each $H$-orbit has size at least $1+s$ by Lemma \ref{lem_Hprop}, we deduce that $s<\frac{3}{2}(q-1)$. Since $1+s^2$ is an odd multiple of $q^2$ by Lemma \ref{lem_Greg_coll} and $1+\frac{9}{4}(q-1)^2<3q^2$, we deduce that $1+s^2=q^2$, which is impossible. This completes the proof.
\end{proof}

In view of Lemma \ref{lem_prim_1dim}, we identify $N$ as the additive group of the field $\F_{p^d}$, where $p^d=|N|$.  Take $\gamma$ to be a primitive element of $\F_{p^d}$. Let $\rho$ be the element of $\textup{GL}_1(p^d)$ such that $\rho(x)=\gamma x$ for $x\in \F_{p^d}$.  We set $e:=[H:\,H_0]$, where 
\begin{equation}\label{eqn_1dim_H0def}
  H_0:=H\cap\textup{GL}_1(p^d).
\end{equation}
\noindent Since $|G|$ is odd, we deduce that $e$ is odd. We have $H_0=\la \rho^k\ra$ for some divisor $k$ of $p^d-1$, and $p^d=q_1^e$ for some prime power $q_1$. We have $H=\la \rho^k,\rho^{s_0}\sigma\ra$ for some integer $s_0$, where $\sigma(x)=x^{q_1}$ for $x\in\F_{p^d}$. It follows that $H'=\la\rho^{k(q_1-1)}\ra$, and it is contained in $H_0$. 

\begin{lemma}\label{lem_1dim_Hprop}
Take notation as above. Then
\begin{enumerate}
\item[(1)] $|H|>1+s$, and $|C_G(N)|\le s^2$.
\item[(2)] $H$ is nonabelian, and $d\ge e>1$.
\item[(3)] $|H'|$ divides $\frac{q_1^e-1}{q_1-1}$, and $[H:\,H']$ divides $\frac{1}{4}e(q_1-1)$, where $p^d=q_1^e$.
\end{enumerate}
\end{lemma}
\begin{proof}
(1) It suffices to show that $|H|>1+s$, since then $|C_G(N)|=\frac{|G|}{|H|}<1+s^2$. Suppose to the contrary that $|H|\le1+s$. By Lemma \ref{lem_Hprop}, $|H|=1+s$ and so each $H$-orbit on $N\setminus\{1\}$ must have size $1+s$. It follows that $1+s$ divides $p^d-1$. There exists an odd integer $\lambda$ such that $1+s^2=\lambda p^d$ by Lemma \ref{lem_Greg_coll} (2). Taking modulo $s+1$, we obtain $\lambda-2\equiv 0\pmod{1+s}$. Since $\lambda-2$ is odd, we deduce that $\lambda-2\ge 1+s$. On the other hand, $p^d\ge 2s+3$ by Lemma \ref{lem_Hprop}, so $\lambda<\frac{1+s^2}{2s+2}<1+s$: a contradiction. This proves the claim.

(2) Since $H_0$ is abelian and $|H|=e\cdot|H_0|$, it suffices to show that $H$ is nonabelian. Suppose to the contrary that $H$ is abelian, i.e., $H'=1$.  It follows from $H'=\la\rho^{k(q_1-1)}\ra$ that $H_0=\la\rho^k\ra$ has order dividing $q_1-1$, and so $|H|\le e(q_1-1)$. We have $|H|>|N|^{1/2}=q_1^{e/2}$ by Lemma \ref{lem_Hprop}, so $q_1^{e/2}<e(q_1-1)$. The latter inequality holds with $q_1\equiv 1\pmod{4}$ and $e$ odd only if $e=1$ or $(q_1,e)=(5,3)$. In both cases, we deduce that $|H|$ is relatively prime to $|N|$. Since $H'=1$ and $H=G/C_G(N)$, we have $G'\le C_G(N)$.  Therefore, $|H|=[G:\,C_G(N)]$ divides $[G:\,G']$. By (1) and Lemma \ref{lem_Gbarorder_seqt}, we deduce that $|H|$ divides $1+s^2$. Since $|N|$ divides $1+s^2$ by Lemma \ref{lem_Greg_coll}, we deduce that $|H|\cdot|N|$ divides $1+s^2$. On the other hand, $|H|\cdot|N|\ge (1+s)(2s+3)>1+s^2$ by Lemma \ref{lem_Hprop}: a contradiction. This proves the claim.

(3) We have $H'=\overline{G'}=\la \rho^{k(q_1-1)}\ra$, so its order divides $\frac{q_1^e-1}{q_1-1}$ and it is contained in $H_0$. The second claim follows from the facts $[H:\,H_0]=e$, $[H_0:\,H']$ divides the odd part of $q_1-1$, and $p\equiv1\pmod{4}$ by Lemma \ref{lem_Greg_coll}.
\end{proof}

\begin{lemma}\label{lem_N_line}
For each $g\in N\setminus\{1\}$, it fixes no line, and
\begin{equation}\label{eqn_C4_gGdel}
   |g^G\cap \Delta|<\frac{2+\sqrt{2s}}{1+s}\cdot|g^G|,
\end{equation}
where $\Delta$ is as in \eqref{eqn_Delta_def}. Moreover, we have
\begin{equation}\label{eqn_C4_gGdelN}
   |N\cap\Delta|<\frac{2+\sqrt{2s}}{1+s}\cdot (p^d-1).
\end{equation}
\end{lemma}
\begin{proof}
Take $g\in N\setminus\{1\}$, and suppose that it fixes a line $\ell$. Let $m$ be the order of $g$. Its orbits on the points of $\ell$ all have lengths $m$, since $G$ is regular on points. It follows that $m$ divides $1+s$. On the other hand, $|N|$ divides $1+s^2$ by Lemma \ref{lem_Greg_coll}, so $m$ divides $\gcd(1+s,1+s^2)=1$: a contradiction. This proves the first claim. By Lemma \ref{lem_Y_Benson}, we have $|\cL_2(g)|=|C_G(g)|\cdot|g^G\cap\Delta|$, where $\cL_2(g)$ is as in \eqref{eqn_reg_P2L2def}. The inequality \eqref{eqn_C4_gGdel} then follows from the upper bound on $|\cL_2(g)|$ in Theorem \ref{thm_l2g_upp} and the fact $|C_G(g)|\cdot|g^G|=(1+s)(1+s^2)$. Taking summation over the nontrivial conjugacy classes of $G$ contained in $N$, we deduce \eqref{eqn_C4_gGdelN} from \eqref{eqn_C4_gGdel}. This completes the proof.
\end{proof}

\begin{lemma}\label{lem_1dim_sbound}
We have $s\le d^{-1}p^d$.
\end{lemma}
\begin{proof}
Suppose to the contrary that $s>d^{-1}p^d$. Since $|H|$ divides $|\GamL_1(p^d)|=d(p^d-1)$ and $|H|>1+s$ by Lemma \ref{lem_1dim_Hprop}, we deduce that $s<dp^d$. There are integers $a,b$ such that
\[
  p^d-1=as+b,\quad 0\le b\le s-1.
\]
Then $p^d>as>ad^{-1}p^d$, so $a<d$. Recall that $p\equiv 1\pmod{4}$  and $1+s^2\equiv 0\pmod{p^d}$ by Lemma \ref{lem_Greg_coll}. We divide the proof into four steps.

(i) We first show that $d^3(2+\sqrt{2dp^d})<p^d$. It involves only $(p,d)$ and is violated only if $p=5$ and $d\le 10$, or $(p,d)=(13,3)$.  There is no even integer  $s$ such that $d^{-1}p^d<s<dp^d$, $1+s^2\equiv 0\pmod{p^d}$, and $\gcd((1+s)(1+s^2),d(p^d-1))>1+s$ in each case, so $(p,d)$ is none of those pairs. This establishes the claim.

(ii) We next show that $b=|N\cap\Delta|$ and $b<d^{-1}s$, where $\Delta$ is defined as in \eqref{eqn_Delta_def}. We deduce from \eqref{eqn_C4_gGdelN} that
\begin{equation}\label{eqn_tp1}
\frac{|N\cap\Delta|}{s}< \frac{2+\sqrt{2s}}{s(1+s)}(p^d-1)
<\frac{2+\sqrt{2dp^d}}{s^2}p^d<d^2(2+\sqrt{2dp^d})p^{-d}<\frac{1}{d},
\end{equation}
where we used (i) in the last inequality. Therefore, $|N\cap\Delta|<d^{-1}s$. We deduce that  $|N\cap\Delta|\equiv p^d-1\equiv b\pmod{s}$ from Proposition \ref{prop_gGcapDel}. The claim then follows.

(iii) Take an element $g$ of $N\setminus\{1\}$. Since $|H|=\frac{p^d-1}{k}\cdot e$ and $H_0$ is semiregular on $N\setminus\{1\}$, there is a divisor $e'$ of $e$ such that $|g^G|=\frac{p^d-1}{k}\cdot e'$, i.e., $|g^G|=\frac{ae'}{k}s+\frac{be'}{k}$. We deduce that $\frac{k-1}{k}s+\frac{be'}{k}<s$, i.e., $be'<s$, from (ii) and the fact $e'\le d$.  We claim that $k$ divides $e'\cdot\gcd(a,b)$. As a corollary, $|g^G\cap\Delta|=\frac{be'}{k}$ and is the remainder of $|g^G|$ modulo $s$ by Proposition \ref{prop_gGcapDel} and the fact $|g^G\cap\Delta|\le |g^G\cap N|=b<s$. It suffices to show that $k$ divides $ae'$. If the remainder $r$ of $ae'$ modulo $k$ is nonzero, then $\frac{r}{k}s+\frac{be'}{k}\le \frac{k-1}{k}s+\frac{be'}{k}<s$, and so $|g^G|\mod{s}$ equals $\frac{r}{k}s+\frac{be'}{k}$ which is at least $k^{-1}s$. By Proposition \ref{prop_gGcapDel}, we have $|g^G\cap\Delta|\ge k^{-1}s$. On the other hand, $|g^G\cap\Delta|<\frac{1}{ks}(2+\sqrt{2s})dp^d$ by \eqref{eqn_C4_gGdel} and the fact $|g^G|<k^{-1}dp^d$. By combining the two inequalities, we deduce that $s^2<(2+\sqrt{2s})dp^d$. Since $d^{-1}p^d<s<dp^d$, it follows that $d^{-2}p^{2d}<(2+\sqrt{2dp^d})dp^d$, i.e.,  $p^{d}<(2+\sqrt{2dp^d})d^3$. This contradicts (i) and establishes the claim.

(iv) Take an element $g\in N\setminus\{1\}$. There is a divisor $e'$ of $e$ such that $k$ divides $e'\gcd(a,b)$, $|g^G|=\frac{p^d-1}{k}e'$ and $|g^G\cap \Delta|=\frac{be'}{k}$ by step (ii). Since $p^d=as+b+1$ and $p^d$ divides $1+s^2$, we have
$0\equiv a^2(1+s^2)\equiv a^2+(b+1)^2\pmod{p^d}$, so $a^2+(b+1)^2=up^d$ for some integer $u$.  From the facts $a\cdot|g^G|=(p^d-1)\cdot\frac{ae'}{k}$ and $|g^G|$ divides $|G|$, we deduce that $a\cdot|G|$ is a multiple of $p^d-1$. Therefore,
\begin{align}
0\equiv &a^3\cdot|G|=(a+as)(a^2+a^2s^2)\equiv(a-b)(a^2+b^2)=(a-b)(up^d-1-2b)\notag\\
\equiv& (a-b)(u-1-2b)\equiv (a-b)(u-1)-2ab+2(u-1-2b-a^2)\notag\\
=&-(u+3+2a)b+(a+2)(u-1)-2a^2\pmod{p^d-1}.\label{eqn_tp2_cong}
\end{align}
The last term in \eqref{eqn_tp2_cong} is nonzero, since otherwise $b=a+2-\frac{4a^2+8a+8}{u+3+2a}<a+2$ and we would derive a contradiction $a^2+(b+1)^2<d^2+(d+2)^2<p^d$ by the fact $a<d$. We conclude that $-(u+3+2a)b+(a+2)(u-1)-2a^2$ is a nonzero multiple of $p^d-1$, which implies that
\begin{equation}\label{eqn_1dim_tp3}
(u+3+2a)b+(a+2)(u-1)+2a^2\ge p^d-1.
\end{equation}
By \eqref{eqn_C4_gGdelN}, we have $b=|N\cap\Delta|<(2+\sqrt{2s})s^{-1}p^d<d(2+\sqrt{2dp^d})$.
It follows that $u=\frac{a^2+(b+1)^2}{p^d}$ is upper bounded by $2d^3+1$ when $p^d$ is sufficiently large. We plug the expression of $u$ in \eqref{eqn_1dim_tp3} and use the bounds $a<d$, $b<d(2+\sqrt{2dp^d})$ to obtain a new inequality that involves only $p$ and $d$. There are $31$ such $(p,d)$ pairs that $p\equiv 1\pmod{4}$, $d\ge 3$ and the new inequality holds. Those cases are excluded by showing that there is no even integer $s$ with desired properties as in step (i). This completes the proof.
\end{proof}

\begin{lemma}\label{lem_H0_irred}
The group $H_0$ is irreducible on $\F_q$, where $N=(\F_q,+)$.
\end{lemma}
\begin{proof}
Suppose to the contrary that $H_0=\la \rho^k\ra$ is reducible on $V$. This is the case if and only if $\F_p[\gamma^k]=\F_{p^u}$ for some proper divisor $u$ of $d$. Since $H'\ne1$ by Lemma \ref{lem_1dim_Hprop}, we have $\gamma^{k(q_1-1)}\ne 1$. This implies that $\gamma^k$ does not lie in $\F_{q_1}$, i.e., $\F_{p^u}$ is not a subfield of $\F_{q_1}$. In other words, $u$ does not divide $\frac{d}{e}$, since $q_1=p^{d/e}$. By Lemma \ref{lem_Hprop}, we have $|H|=e\cdot |H_0|>q^{1/2}$. Since $|H_0|\le p^u-1$, we deduce that $e(p^u-1)>p^{d/2}$. We have $e\ge 3$ by Lemma \ref{lem_1dim_Hprop} (2).

We claim that $d$ is even and $u=d/2$. If $d$ is odd, then $u\le d/3$ and the only $(p,d)$ pair such that $p\equiv 1\pmod{4}$ and $d(p^{d/3}-1)>p^{d/2}$ is $(5,3)$. In the case $(p,d)=(5,3)$, we deduce that $u=1$ and $e=3$, but it contradicts the fact that $u$ does not divide $\frac{d}{e}$. We thus conclude that $d$ is even, and so $e\le d/2$. If $u\ne d/2$, then $u\le d/3$,  and  $e(p^u-1)>p^{d/2}$ implies that $\frac{d}{2}>p^{d/2-u}\ge p^{d/6}$. The latter inequality holds for no $(p,d)$ pair with $p\equiv 1\pmod{4}$ and $d\ge 6$ even. This proves the claim.

There is an odd integer $m$ such that $1+s^2=mp^d$ by Lemma \ref{lem_Greg_coll}. We have $|H|>1+s$ by Lemma \ref{lem_1dim_Hprop} (1), so $1+s^2<|H|^2$. Since $|H|$ divides $e(p^{d/2}-1)$,  we deduce from $1+s^2<|H|^2$ that $m<e^2$. Since $|H|\cdot|N|$ divides $|G|=(1+s)(1+s^2)$, it follows that $|H|$ divides $m(1+s)$. In particular, we have $m\ge 3$ by the fact $|H|>1+s$.  Taking modulo $|H|$, we have
\[
  2e m^2\equiv e\cdot m^2\cdot (1+s^2)=m^3ep^d\equiv m^3e\pmod{|H|}.
\]
It follows that $|H|$ divides $m^2e(m-2)$. Since $3\le m<e^2$, we deduce that $e^5(e^2-2)>|H|>p^{d/2}$, where the last inequality is from Lemma \ref{lem_Hprop}. Since $d$ is even and $e\ge 3$ is odd, we have $d/2\ge e$. It follows that $e^5(e^2-2)>p^{e}$.   This holds only if $p=5$, and $e$ is one of $3,5,7,9$. In each case, $e^5(e^2-2)>p^{d/2}$ holds only if $d=2e$. It follows that $u=e$. We have $1+s<|H|\le e(p^{e}-1)<p^d$, and there is a unique even integer $s$ such that $1+s^2\equiv 0\pmod{p^d}$, $0\le s\le p^d-1$. For each odd integer $m$ with $3\le m<e^2$, we check that $D:=\gcd(m(1+s),e(p^{e}-1))$ is smaller than $1+s$ in each case, contradicting the facts that $|H|>1+s$ and $|H|$ divides $D$. This completes the proof.
\end{proof}

We shall need the following elementary fact from linear algebra.
\begin{proposition}\label{prop_irred_gord}
If an element $g\in\textup{GL}_m(p)$ is irreducible on $\F_p^m$, then its order divides $p^m-1$ but does not divide $p^a-1$  for all $1\le a\le m-1$.
\end{proposition}
\begin{proof}
Let $f(x)$ be the characteristic polynomial of $g$. By \cite[p.~119]{suprk}, $f(x)$ is irreducible over $\F_p$. It divides $x^{p^m-1}-1$ by \cite[Corollary~3.4]{FF}, which has no repeated root. Therefore, $g$ can be diagonalized over $\F_{p^m}$ which contains all roots of $f(x)$.  The order of $g$ equals that of any root $\alpha$ of $f(x)$. Since $f$ is irreducible of degree $m$, its root $\alpha$ lies in $\F_{p^m}$ and no smaller extension field of $\F_p$. The claim then follows.
\end{proof}

For an element $g$ of a finite group, we use $o(g)$ for its order.
\begin{lemma}\label{lem_1dim_NmaxAbe}
The subgroup $N$ is a maximal abelian normal subgroup of $G$.
\end{lemma}
\begin{proof}
Suppose to the contrary that $M$ is an abelian normal subgroup of $G$ that properly contains $N$. It is elementary abelian by Lemma \ref{lem_1dim_abelian}.  Take an element $g$ of $G$ such that $H_0=\la\bar{g}\ra$, where $H_0$ is as in \eqref{eqn_1dim_H0def}. By replacing $g$ with a proper power $g^i$ if necessary, we assume without loss of generality that all prime divisors of $o(g)$ divide $|H_0|$. In particular, $o(g)$ is relatively prime to $p$.

The action of $\la g\ra$ on $M$ via conjugation is semisimple, and $N$ is an irreducible $\la g\ra$-submodule of $M$ by Lemma \ref{lem_H0_irred}. By Maschke's Theorem \cite[p.~66]{Gor}, we have a decomposition of $M$ as a direct sum of irreducible $\la g\ra$-submodules with $N$ as an constituent. Take $W$ to be a constituent distinct from $N$, and write $m:=\dim_{\F_p}(W)$. Let $g_W$ be the induced action of $g$ on $W$, so that $g_W$ is irreducible on $W$. By Proposition \ref{prop_irred_gord}, $o(g_W)$ divides $p^m-1$. Therefore, $g^{p^m-1}$ acts trivially on $W$, i.e., lies in the centralizer of $W$. Its order is at most $\frac{o(g)}{p^m-1}$. Take a nonidentity element $v$ of $W$. It is centralized by $M$, which has order at least $|N|\cdot|W|=p^{d+m}$. Hence $C_G(v)$ has order at least $p^{d+m}\cdot\frac{o(g)}{p^m-1}$, and so
\[
|v^G|\le \frac{(1+s)(1+s^2)(p^m-1)}{p^{d+m}\cdot o(g)}<\frac{(1+s)(1+s^2)}{p^{d}\cdot d^{-1}(1+s)}=dp^{-d}(1+s^2).
\]
Here we have used the facts $o(g)\ge o(\bar{g})$ and $o(\bar{g})=|H_0|\ge d^{-1}|H|>d^{-1}(1+s)$. Since $|v^G|\ge 1+s$ by Proposition \ref{prop_gGcapDel}, we have $1+s<dp^{-d}(1+s^2)$. It follows that $s>\frac{1+s^2}{1+s}>d^{-1}p^d$, which contradicts Lemma \ref{lem_1dim_sbound}. This completes the proof.
\end{proof}

\begin{lemma}
We have $F(G)=N$.
\end{lemma}
\begin{proof}
We write $F:=F(G)$, and suppose to the contrary that $N\ne F$. By Lemmas \ref{lem_N_minn} and \ref{lem_1dim_NmaxAbe}, we deduce that $N$ is the center of $F$. Therefore, $F$ is a nonabelian $p$-group by the assumption $N\ne F$, and $F$ is a subgroup of $C_G(N)$. Let $L$ be a minimal nontrivial $G$-invariant subgroup of the center of $F/N$, which must be elementary abelian. Let $M$ be the full preimage of $L$ in $F$. It is clear that $M'\le [F,M]\le N$ by the choice of $L$. The group $M$ is nonabelian by Lemma \ref{lem_1dim_NmaxAbe}, so $M'=N$ by Lemma \ref{lem_N_minn}. Since $N$ centralizes $M$, there is a well-defined map
\begin{equation}\label{eqn_psiDef}
\psi:\,L\times L\rightarrow N,\quad (\bar{x},\bar{y})\mapsto [x,y],
\end{equation}
where $x,y\in M$ and $\bar{x},\bar{y}$ are their images in $L$ respectively. It commutes with the action of $g$, i.e., $\psi(\bar{x}^g,\bar{y}^g)=\psi(\bar{x},\bar{y})^g$. The subgroup of $N$ spanned by the image set of $\psi$ is $M'=N$.

Suppose that $|L|=p^m$. We claim that $d^2p^m\le p^d$, and $m<d$. We have $|N|\cdot |L|\le |C_G(N)|\le s^2$ by Lemma \ref{lem_1dim_Hprop} (1). Since $s\le d^{-1}p^d$ by Lemma \ref{lem_1dim_sbound},  we deduce that $p^{m+d}\le d^{-2}p^{2d}$, i.e., $d^2p^m\le p^d$. The second claim then follows from the fact $d>1$.

Take an element $g$ of $G$ such that $H_0=\la\bar{g}\ra$ and all prime divisors of $o(g)$ divide $|H_0|$ as in the proof of Lemma \ref{lem_1dim_NmaxAbe}. The element $\bar{g}$ is irreducible on $N$ by Lemma \ref{lem_H0_irred}. There is an induced action of $g$ on $L$ via conjugation, which makes $L$ into a $\la g\ra$-module. By Maschke's Theorem, there is a decomposition $L=L_1\oplus\cdots\oplus L_t$ of $L$ into the direct sum of irreducible $\la g\ra$-submodules. For $1\le i\le t$, we write $|L_i|=p^{m_i}$ and use $g_i$ for the induced action of $g$ on $L_i$. By Proposition \ref{prop_irred_gord}, $o(g_i)$ divides $p^{m_i}-1$ for each $i$. We have $m=\sum_{i=1}^tm_i<d$.

For each $(i,j)$ pair, $N_{ij}:=\la\psi(a,b):\,a\in L_i, b\in L_j\ra$ is a $\la g\ra$-invariant submodule of $N$. Since $N$ is irreducible, it is either $1$ or $N$. If $N_{ii}=N$ for some $i$, then the action of $g$ on $N$ is induced from its action on $L_i$. It follows that $o(\bar{g})$ divides $o(g_i)$. This implies that $o(\bar{g})$ divides $p^{m_i}-1$, which contradicts Proposition \ref{prop_irred_gord}. Therefore, $N_{ii}=1$ for each $i$, and so $t\ge 2$ and there is a pair of distinct $i,j$ such that $N_{ij}=N$. With proper relabeling, we assume without loss of generality that $N_{12}=N$.

The action of $g$ on $N$ is induced from its action on $L_1$ and $L_2$, so $o(\bar{g})$ divides the least common multiple of $o(g_1),\,o(g_2)$. In particular, it divides $D:=\lcm(p^{m_1}-1,p^{m_2}-1)=\frac{(p^{m_1}-1)(p^{m_2}-1)}{p^r-1}$, where $r=\gcd(m_1,m_2)$.
Since $|H|>1+s$, $|H|=|H_0|\cdot e$ and $H_0=\la o(\bar{g})\ra$, it follows that $1+s< dD$. The subgroup $F$ is in $C_G(N)$ and contains $M$, so $|M|\le |C_G(N)|\le s^2$ by Lemma \ref{lem_1dim_Hprop} (1). It follows that $p^{m_1+m_2+d}\le s^2<d^2D^2$, which implies that $p^d(p^r-1)^2<d^2p^{m_1+m_2}$. Since $d^2p^m\le p^d$ and $m_1+m_2\le m$, it leads to the contradiction $(p^r-1)^2<1$. This completes the proof.
\end{proof}

\begin{lemma}\label{lem_NeqK}
We have $C_G(N)=N$.
\end{lemma}
\begin{proof}
We write $K:=C_G(N)$, and suppose that $K\ne N$. Let $n$ be the derived length of $K/N$, and define $D_0(K/N)=K/N$ and $D_{i+1}(K/N)=D_i(K/N)'$ for $i\ge 0$. Let $L$ be a nontrivial $G$-invariant subgroup of $K/N$ of smallest order contained in $D_{n-1}(K/N)$, and let $M$ be its full preimage in $K$. Then $L$ is an elementary abelian $r$-group for some prime $r$. We have $r\ne p$, since otherwise $M$ would be a normal $p$-subgroup of $G$ and thus contained in $F(G)=N$.  The group $M$ is nonabelian by Lemma \ref{lem_1dim_NmaxAbe}. Since $L$ is abelian, $M'\le N$. It follows that $M'=N$ by the minimality of $N$. For $x,y\in M$, we have $y^r\in N$, $x,y\in C_G(N)$, so $1=[x,y^r]=[x,y]^r$ by the commutator formula $[x,yz]=[x,z][x,y]^z$, cf.  \cite[p.~18]{Gor}. Since $[x,y]$ is in the elementary abelian $p$-group $N$ and $r\ne p$, we deduce that $[x,y]=1$ for $x,y\in M$. This contradicts the fact $M'=N$ and completes the proof.
\end{proof}

\begin{lemma}\label{lem_1dim_Gdd}
We have $G''=N$.
\end{lemma}
\begin{proof}
The group $H=G/C_G(N)$ is nonabelian by Lemma \ref{lem_Hprop} (2), so $G'$ is not contained in $C_G(N)=N$. By Lemma \ref{lem_N_minn}, $N<G'$. The group $G'$ is nonabelian by Lemma \ref{lem_1dim_NmaxAbe}, so $G''>1$. We have $N\le G''$ by Lemma \ref{lem_N_minn}. Since $H$ is a subgroup of $\GamL_1(p^d)$, we deduce that $H''=1$, i.e., $G''\le N$. The claim $G''=N$ then follows.
\end{proof}

\begin{lemma}\label{lem_1dim_GGd}
$[G:\,G']$ divides $1+s$.
\end{lemma}
\begin{proof}
By Lemma \ref{lem_Gbarorder_seqt}, $[G:\,G']$ divides $1+s$ or $1+s^2$. Suppose to the contrary that  $[G:\,G']$ divides  $1+s^2$. Since $G''=N$ by Lemma \ref{lem_1dim_Gdd}, we have $|G|=[G:\,G']\cdot [G':\,N]\cdot p^d=(1+s)(1+s^2)$. Since $p^d$ divides $1+s^2$ by Lemma \ref{lem_Greg_coll}, we deduce that $1+s$ divides $|H'|=[G':\,N]$. Also, $|H'|$ divides $p^d-1$ by Lemma \ref{lem_1dim_Hprop} (3). Write $p^d-1=\alpha(1+s)$ for some integer $\alpha$. We have
\begin{align*}
0\equiv \alpha^2(1+s^2)&=\alpha^2+(p^d-1-\alpha)^2\equiv 2\alpha^2+2\alpha+1\pmod{p^d}.
\end{align*}
In particular,  $2\alpha^2+2\alpha+1\ge p^d$. It follows that $\alpha\ge\frac{-1+\sqrt{2p^d-1}}{2}$, and so $s=\frac{p^d-1}{\alpha}-1\le\sqrt{2p^d-1}$.
It is equivalent to $p^d\ge\frac{s^2+1}{2}$. Since $s^2+1$ is odd and $p^d$ divides $s^2+1$, we deduce that $p^d=s^2+1$. Then $\alpha=\frac{p^d-1}{s+1}=\frac{s^2}{s+1}$ which is not an integer: a contradiction. This completes the proof.
\end{proof}

We are now ready to complete the proof of Theorem \ref{thm_regGQ}. Recall that $|G|=(1+s)(1+s^2)$, $|N|=q=q_1^e$, $H=G/C_G(N)$, $H=\la\rho^k,\rho^{s_0}\sigma\ra$, $H'=\la\rho^{k(q_1-1)}\ra$, where $e$ is an odd integer. Set $\beta:=\frac{q_1^e-1}{q_1-1}$.  By Lemma \ref{lem_Hprop}, $e>1$, and $|H'|$ divides $\beta$. By Lemmas \ref{lem_NeqK} and \ref{lem_1dim_Gdd}, $C_G(N)=G''=N$. It follows that $H'=G'/N$ and $[G:\,G']=[H:\,H']$, $|H'|=[G':\,N]$. Set $u:=[G:\,G']$ and $u_0:=\frac{1}{e}u$. Then $u_0$ is an odd integer that divides $\frac{1}{4}(q_1-1)$ by Lemma \ref{lem_1dim_Hprop} (3).  By Lemmas \ref{lem_Greg_coll} and \ref{lem_1dim_GGd}, $p^d$ divides $1+s^2$, and $u$ divides $1+s$.  Since $|G|=u\cdot |H'|\cdot q=(1+s)(1+s^2)$, there are odd divisors $n_1,n_2$ of $|H'|$ such that
\[
  1+s=n_1u,\quad 1+s^2=qn_2, \quad n_1n_2=|H'|,\quad \beta\equiv 0\pmod{n_1n_2}.
\]
We deduce from the first two equations that $1+(en_1u_0-1)^2=qn_2$. Taking modulo $n_1u_0$, we deduce that $n_2-2\equiv 0\pmod{n_1u_0}$. Here we have used the fact that $n_1u_0$ divides $q-1$. Since both $n_1u_0$ and $n_2$ are odd, it follows that  $n_2\ge n_1u_0+2$.  Therefore, $1+(en_1u_0-1)^2\ge q(n_1u_0+2)$. In particular, we have $e^2n_1^2u_0^2>qu_0n_1$, i.e., $n_1>r$ with $r=\frac{q}{e^2u_0}$. We deduce from $n_1n_2\le\beta$ that $n_1^2u_0+2n_1\le\beta$, so $r^2u_0<\beta$. On the other hand, since $u_0$ divides $\frac{q_1-1}{4}$, we have
\begin{align*}
r^2u_0-\beta&=\frac{q^2}{e^4u_0}-\frac{q-1}{q_1-1} >\frac{4q^2}{e^4(q_1-1)}-\frac{q}{q_1-1}.
\end{align*}
Since $4q_1^e>e^4$ for all $(q_1,e)$ pairs with $q_1\equiv 1\pmod{4}$ and  $e\ge 3$ odd, the last term is always positive. This contradiction completes the proof of Theorem \ref{thm_regGQ}.\\

\noindent\textbf{Acknowledgements}. This work was supported by National Natural Science Foundation of China under Grant No. 12225110, 12171428 and the Sino-German Mobility Programme M-0157. The author thanks the reviewers for detailed comments and suggestions that helped to improve the presentation of this paper. The author thanks Bill Kantor for pointing out the reference \cite{suprk} which considerably simplified the original proof of Lemma \ref{lem_prim_1dim}.

\end{document}